\documentclass[11pt]{amsart}
\usepackage{amsmath, amsfonts, amssymb,amsthm, amscd}

\usepackage{latexsym,graphics}
\usepackage{epsfig,epsf,psfrag,color}

\usepackage{graphicx}
\usepackage{epsf}
\usepackage{epsfig}
\usepackage{verbatim}
\usepackage{psfrag}
\usepackage{color}
\usepackage{hyperref}

\definecolor{blue}{rgb}{0,0,1}
\definecolor{red}{rgb}{1,0,0}
\definecolor{green}{rgb}{0,.7,0}
\newtheorem{theorem}{Theorem}[]
\newtheorem{lemma}[theorem]{Lemma}

\newtheorem{corollary}[]{Corollary}

\theoremstyle{definition}
\newtheorem{definition}[]{Definition}

\theoremstyle{remark}
\newtheorem{remark}[]{Remark}

\newtheoremstyle{thm}
  {12pt}
  {12pt}
  {\itshape}
  {\parindent}
  {\scshape}
  {.}
  {5pt}
  {}

\theoremstyle{thm}
\newtheorem*{T6.17}{Theorem \ref{gradientthm1}}
\newtheorem{thm}{Theorem}[]
\newtheoremstyle{prop}
  {12pt}
  {12pt}
  {\itshape}
  {\parindent}
  {\scshape}
  {.}
  {5pt}
  {}

\theoremstyle{prop}

\newtheoremstyle{lem}
  {12pt}
  {12pt}
  {\itshape}
  {\parindent}
  {\scshape}
  {.}
  {5pt}
  {}

\theoremstyle{lem}
\newtheorem*{L2.6}{Lemma \ref{lem2.4}}

\newtheoremstyle{defn}
  {12pt}
  {12pt}
  {\itshape}
  {\parindent}
  {\scshape}
  {.}
  {5pt}
  {}

\theoremstyle{defn}

\newtheoremstyle{examp}
  {12pt}
  {12pt}
  {}
   {\parindent}
  {\scshape}
  {.}
  {5pt}
  {}

\theoremstyle{examp}

\newtheoremstyle{cor}
  {12pt}
  {12pt}
  {\itshape}
  {\parindent}
  {\scshape}
  {.}
  {5pt}
  {}

\theoremstyle{cor}

\newtheoremstyle{recipe}
  {12pt}
  {12pt}
  {\itshape}
   {\parindent}
  {\scshape}
  {.}
  {5pt}
  {}

\theoremstyle{recipe}

\newtheoremstyle{rem}
  {12pt}
  {12pt}
  {}
   {\parindent}
  {\scshape}
  {.}
  {5pt}
  {}

\theoremstyle{rem}

\newtheoremstyle{quest}
  {12pt}
  {12pt}
  {\itshape}
  {\parindent}
  {\scshape}
  {.}
  {5pt}
  {}

\theoremstyle{quest}
\newtheorem{quest}[thm]{Question}

\newcommand{\bp}{\begin{proof}}
\newcommand{\ep}{\end{proof}}

\renewcommand{\epsilon}{\varepsilon}

\newcommand{\Z}{{\mathbb Z}}

\newcommand{\id}{\operatorname{id}}

\newcommand{\R}{{\mathbb R}}

\providecommand{\ker}[1]{$\text{ker}\ {#1}$}

\newcommand{\N}{{\mathbb N}}
\newcommand{\Q}{{\mathbb Q}}

\newcommand{\T}{\mathcal{T}}
\newcommand{\Rot}{\operatorname{Rot}}

\newcommand{\F}{\mathcal{F}}
\newcommand{\A}{\mathcal{A}}
\newcommand{\Diff}{\operatorname{Diff}}

\renewcommand{\u}{\tilde{u}}
\renewcommand{\v}{\tilde{v}}

\newcommand{\C}{\mathbb{C}}

\renewcommand{\max}{\textup{max}}
\renewcommand{\O}{\mathcal{O}}

\newcommand{\D}{\mathcal{D}}

\newcommand{\U}{\mathcal{U}}

\renewcommand{\L}{\mathcal{L}}

\newcommand{\M}{\mathcal{M}}

\newcommand{\fbar}{\bar{f}}

\renewcommand{\L}{\mathcal{L}}

\renewcommand{\O}{\mathcal{O}}

\renewcommand{\u}{\tilde{u}}
\renewcommand{\v}{\tilde{v}}

\newcommand{\gammadot}{\dot{\gamma}}

\newcommand{\psihat}{\hat{\psi}}
\newcommand{\Hess}{\textup{Hess}}

\newcommand{\Floer}{\textnormal{Floer}}

{\catcode`"=12 \gdef\hex{"}}

\mathchardef\laplace=\hex0001

\mathchardef\nabla=\hex0272

\makeatletter

\def\@@dalembert#1#2{\setbox0\hbox{$#1\mathrm I$}

  \vrule height\ht0 depth\z@ width.04\ht0

  \rlap{\vrule height\ht0 depth-.96\ht0 width.8\ht0}

  \vrule height.1\ht0 depth\z@ width.8\ht0

  \vrule height\ht0 depth\z@ width.1\ht0 }

\def\dalembert{\mathbin{\mathpalette\@@dalembert{}}\,}

\makeatother

\begin{document}
\address{
    Barney Bramham\\
    Institute for Advanced Study, Princeton, NJ 08540}
\email{bramham@ias.edu}

\title[Rigidity of Liouville pseudo-rotations]{Pseudo-rotations with sufficiently Liouvillean rotation number are
$C^0$-rigid.}

\author[Barney Bramham]{Barney Bramham}
\maketitle

\vspace{-0.2in}
\begin{abstract}
It is an open question in smooth ergodic theory whether there exists a Hamiltonian disk map
with zero topological entropy and (strong) mixing dynamics.  Weak mixing has been known
since Anosov and Katok first constructed examples in 1970.  Currently all known examples
with weak mixing are irrational pseudo-rotations with Liouvillean rotation number on the boundary.
Our main result however implies that for a dense subset of Liouville numbers (strong) mixing
cannot occur.
Our approach involves approximating the flow of a suspension of the given disk map by
pseudoholomorphic curves.  Ellipticity of the Cauchy-Riemann
equation allows quantitative $L^2$-estimates to be converted into $C^0$-estimates between the
pseudoholomorphic curves and the trajectories of the flow on growing time scales.  Arithmetic properties
of the rotation number enter through these estimates.
\end{abstract}

\tableofcontents
\section{Introduction}
\subsection{Statement of results}
Pesin theory \cite{Pesin} shows that, at least in low dimensions, smooth conservative systems with
positive metric entropy have ergodicity on sets of positive measure.  It is a delicate question how
``wild'' a smooth system can be if on the other hand it has vanishing metric entropy, or stronger, vanishing
topological entropy.

%

In 1970 Anosov and Katok \cite{AnosovKatok} constructed examples of smooth Hamiltonian disk maps
which, despite having zero topological entropy, were ergodic.  They even found weak mixing examples.
The next level in the ergodic hierarchy, that of mixing\footnote{From here on we use exclusively the term
\emph{mixing} as defined in equation (\ref{E:mixing}).  There are sources in the literature where this notion
is referred to as strong mixing, e.g.\ \cite{Walters}.}, has remained an open question.
In the very interesting article of Fayad and Katok \cite{FayadKatok} this is stated as Problem 3.1
as follows:

\begin{quest}\label{Q:mixing}
 Does there exist a mixing area preserving diffeomorphism of the closed $2$-disk
 $D=\{(x,y)\in\R^2| x^2+y^2\leq1\}$ with zero topological (or metric) entropy?
\end{quest}

Since 1970 the picture has somewhat clarified, and we recall the developments\footnote{In fact the
seeds of this question originate much earlier in 1930 with a construction of
Shnirelman \cite{Shnirelman}.  He found a smooth disk map (not area preserving)
with a dense orbit.  Apparently Ulam was also interested in questions of this nature.  In the Scottish
Book \cite{Scottish}, circa 1935, problem 115, Ulam asks whether there exists a homeomorphism
of $\R^n$ with a dense orbit.  Besicovitch answered this affirmatively for tranformations of
the plane \cite{Bes} in 1951.  Shnirelman's construction is discussed also in \cite{FayadKatok}}.

First, in 1975 Ko$\check{\textup{c}}$ergin \cite{Koc} showed that for any compact surface
of genus $g\geq1$, mixing examples with vanishing topological entropy do exist.  This is
pointed out in \cite{FayadKatok}, where they remark that it is not known for the sphere.

To describe the developments on the disk we recall the following definition.

\begin{definition}
An \emph{irrational pseudo-rotation} is an area and orientation preserving diffeomorphism
$\varphi\in\Diff^\infty(D)$ of the closed $2$-disk fixing the origin and having no other periodic
points.
\end{definition}

Note that the circle map obtained by restricting an irrational pseudo-rotation to the boundary of
the disk automatically has irrational rotation number, hence the terminology.

Irrational pseudo-rotations are the source of all known Hamiltonian disk maps with zero topological entropy
that display weak mixing (or even ergodicity), and is therefore the natural place to look for mixing examples.
Work of Herman and Fayad-Saprykina however refine the search for mixing examples much further.

To describe these results let $\A_{\alpha}$ denote the set of irrational pseudo-rotations
having rotation number $\alpha$ on the boundary of the disk, where $\alpha\in\R/\Z$ is irrational.
Recall that $\alpha\in\R$ is a \emph{Liouville number}
if it is irrational and for all $k\in\N$ there exists $(p,q)\in\Z\times\N$ relatively prime for which
\begin{equation}\label{E:Liouville_condition}
		\left|\alpha - \frac{p}{q}\right|<\frac{1}{q^k}.
\end{equation}
In this paper we will denote the set of all Liouville numbers by $\L$.  These form a dense
set of measure zero in $\R$.
The set of irrational numbers not in $\L$ are called Diophantine, and we will denote these by $\D$.

Anosov and Katok \cite{AnosovKatok} showed that for $\alpha$ in
a dense subset of the Liouville numbers $\A_{\alpha}$ contains a weakly mixing element.
For $\alpha$ a Diophantine number Herman showed in unpublished work, although see also \cite{FayadKrikorian},
that every element of $\A_{\alpha}$ has invariant circles near the boundary of the disk, and in
particular cannot be mixing in any sense.
Finally in 2005 Fayad and Saprykina \cite{Fayad_Sap} constructed weak mixing examples in $\A_{\alpha}$
for \emph{all} $\alpha\in\L$.

To summarize then, the current situation for disk maps is that $\A_{\alpha}$ contains a weak mixing
disk map if and only if $\alpha$ is a Liouville number.

Our results apply to the following subset of the Liouville numbers.  We say that
an irrational number $\alpha\in\R$ is in $\L_*$ if for all $k\in\N$, there exists
$(p,q)\in\Z\times\N$ relatively prime, such that
\begin{equation}\label{E:exponential_Liouville_condition}
				\left|\alpha - \frac{p}{q}\right|<\frac{1}{e^{kq}}.
\end{equation}

\begin{remark}
$\L_*$ is a dense subset of $\L$, see appendix \ref{S:density_L_0}.
\end{remark}

Using pseudoholomorphic curve techniques from symplectic geometry we will prove the following.   

\begin{theorem}\label{T:main_thm}
Let $\varphi:D\rightarrow D$ be an irrational pseudo-rotation with boundary rotation number in
$\L_*$.   Then $\varphi$ is $C^{0}$-rigid.  That is, there exists
a sequence of iterates $\varphi^{n_{j}}$ that converge
to the identity map in the $C^{0}$-topology as $n_{j}\rightarrow\infty$.
\end{theorem}

As an immediate consequence we have: 

\begin{corollary}\label{C:no_mixing}
If $\varphi$ is an irrational pseudo-rotation with its boundary rotation number in $\L_*$ then $\varphi$ is
not mixing.
\end{corollary}

\begin{remark}
In fact such a $\varphi$ is not even topologically mixing.
Recall that a transformation $\varphi:D\rightarrow D$ is topologically mixing if for all open sets
$U,V\subset D$, $\varphi^{-n}(U)\cap V$ is non-empty for all $n$ sufficiently large.  This is implied
by mixing with respect to Lebesgue measure.
\end{remark}

To see how corollary \ref{C:no_mixing} follows from theorem \ref{T:main_thm} we recall the notion of mixing.
Let $\mu$ denote Lebesgue measure on the disk, normalized so that $\mu(D)=1$.
A $\mu$-measure preserving map $\varphi:D\rightarrow D$ is said to be \emph{mixing} if
\begin{equation}\label{E:mixing}
					\lim_{n\rightarrow\infty}\mu(\varphi^{-n}(A)\cap B)=\mu(A)\mu(B)
\end{equation}
for any pair of Lebesgue measurable sets $A,B\subset D$.
In comparison, $\varphi$ is \emph{weak mixing}
if
\begin{equation}\label{E:weak_mixing}
	\lim_{n\rightarrow\infty}\sum_{j=0}^{n-1}\Big|\mu(\varphi^{-j}(A)\cap B)-\mu(A)\mu(B)\Big|=0.
\end{equation}
for all Lebesgue measurable $A,B\subset D$.

\begin{proof}[Proof of corollary \ref{C:no_mixing}]
Let $\varphi$ be an irrational pseudo-rotation with boundary rotation number in $\L_*$.  By theorem \ref{T:main_thm},
$\varphi^{n_j}\rightarrow\id_D$ in the $C^0$-topology for some subsequence $n_j\rightarrow+\infty$.
Therefore also $\varphi^{-n_j}\rightarrow\id_D$.  Take any two disjoint compact measurable sets $A,B\subset D$,
each having non-zero measure.  These are a positive distance apart.  For all $j$ sufficiently large
$\varphi^{-n_j}(A)$ is therefore disjoint from $B$.  On the other hand $\mu(A)\mu(B)>0$.  So the
limit in (\ref{E:mixing}) cannot hold, and therefore $\varphi$ cannot be mixing.
\end{proof}

Theorem \ref{T:main_thm} raises the question whether irrational pseudo-rotations
with \emph{arbitrary} Liouvillean rotation number are $C^{0}$-rigid.  However, naively it would seem 
at least as plausible for Diophantine numbers, since the Liouvillean case typically 
permits ``wilder'' behavior.  This suggests then the following open question.  

\begin{quest}\label{Q:open1}
Is every irrational pseudo-rotation $C^{0}$-rigid?  In other words, if $\varphi:D\rightarrow D$ is any
irrational pseudo-rotation, must there exist a sequence of iterates $\varphi^{n_j}$ converging to the
identity map in the $C^0$-topology?
\end{quest}

Note that answering this question affirmatively for 
Diophantine rotation numbers is a necessary condition for an affirmative answer to the following 
question of Herman, which remains open.  Herman \cite{Herman_ICM} in 1998 asked: 

\begin{quest}\label{Q:open2}
If $\alpha\in\D$ is Diophantine, is every element of $\A_{\alpha}$ smoothly conjugate to
the rigid rotation $R_{2\pi\alpha}$?
\end{quest}

Fayad and Krikorian \cite{FayadKrikorian} have successfully answered Herman's question
for irrational pseudo-rotations sufficiently close, in a differentiable sense, to the
rigid rotation with the same rotation number.

%
%

\subsection{Using a PDE to approximate an ODE}
The idea of the proof of theorem \ref{T:main_thm} can be described well on the level of partial and ordinary
differential equations.

Let $\varphi\in\Diff^{\infty}(D,\omega_{0})$ be an irrational pseudo-rotation and
$\alpha\in\R/\Z$ the rotation number on the boundary.  Let
$H^{t}\in C^{\infty}(D,\R)$ be a closed loop of Hamiltonians generating $\varphi$, over $t\in\R/\Z$.
Denote the corresponding Hamiltonian vector fields on $D$ by $X_{H^{t}}$.
Informally we find a sequence of
``approximating'' (also time-dependent) vector fields on the disk, which converge in $C^{0}$,
\begin{equation}\label{E:convering_vfields}
			X_{n}^{t}\rightarrow X_{H^{t}}
\end{equation}
as $n\rightarrow\infty$.
The trajectories of $X_{n}^{t}$ will close up in time $n$ by construction.  If the convergence
is sufficiently fast this behavior will be reflected in the trajectories of $X_{H^{t}}$, and
the identity map will be an accumulation point for the sequence $\{\varphi^{n}\}_{n\in\N}$.

For each $n\in\N$ the vector field $X_{n}^{t}$ arises as follows.
One can ``fill'' the disk by a family $\M_{n}$ of solutions $z:\R^{+}\times\R/n\Z\rightarrow D$
to the Floer equation
\begin{equation}\label{E:8734}
		\partial_{s}z(s,t)+i\Big(\partial_{t}z(s,t)-X_{H^{t}}(z(s,t))\Big)=0.
\end{equation}
These solutions arise from a foliation of the $4$-manifold $\R\times \R/n\Z\times D$ by pseudoholomorphic
curves constructed in \cite{Bramham_approx}.  For each $t$ a unique solution passes through every point
of the disk, and informally we can define a time-dependent vector field
$\{X_{n}^{t}\}_{t\in\R/\Z}$ on $D$, by declaring
\[
						X_{n}^{t}(z_{n}(s,t)):=\partial_{t}z_{n}(s,t)
\]
for all the solutions $z_{n}$ in $\M_{n}$.  The trajectories of this vector
field are all $n$-periodic because of the cylindrical domain of each $z_{n}$.
The convergence in (\ref{E:convering_vfields}) arises as follows.
We can suggestively rewrite equation (\ref{E:8734}) as
\begin{equation}\label{E:8735}
		\partial_{s}z_n+i\big(X_{n}^t(z_n)-X_{H^t}(z_n)\big)=0.
\end{equation}
It turns out that for such a solution $z_n$, there is a nice expression for the $L^2$-norm of
the first term in (\ref{E:8735}).  Indeed, if we arrange for the right boundary behavior for the
solutions $z_{n}$ then
\[
  \int_{s=0}^{\infty}\int_{t=0}^{n}\left|\frac{\partial z_n}{\partial s}(s,t)\right|^{2}dsdt=\{n\alpha\}\pi,
\]
where $\{n\alpha\}$ denotes the fractional part of $n\alpha\in\R$.
Since $\alpha$ is irrational there exists a subsequence $\{n_j\alpha\}\rightarrow0$ as
$j\rightarrow\infty$.  Thus, the left-most term
in (\ref{E:8734}) decays to zero uniformly in an $L^2$-sense for such a subsequence.
Using the ellipticity of the equation this can be strengthened to convergence in an
$L^{\infty}$-sense (in fact in a $C^{\infty}$-sense but we do not use this).  From (\ref{E:8735})
it follows that $|X_{n_j}^t-X_{H_t}|_{L^{\infty}}\rightarrow0$ as $j\rightarrow\infty$.

The Liouvillean condition on the rotation number $\alpha$ enters the estimates because the
rate of convergence on growing time scales depends on the rate at which $\{n_j\alpha\}\rightarrow0$.

\subsection{Acknowledgements}
I am greatly indepted to Helmut Hofer for many valuable discussions regarding the result in this
article, the ideas that lead up to it, and for helpful comments on
earlier versions of this paper.  I would also like to thank Anatole Katok for a number of very helpful
discussions.  His generous expertise in this area was greatly appreciated and
helped my understanding of many questions that relate to this article.  I also thank Federico Rodriguez Hertz
for discussions with him on a visit to Penn State, and Vadim Kaloshin for helpful conversations at Maryland.

On a visit to the ENS Paris several people spent generous amounts of their time, especially Alain Chenciner,
Jacques F\'ejoz, and Fr\'ed\'eric Le Roux.  The ensuing valuable discussions were much appreciated.
And I would particularly like to thank Claude Viterbo for the invitaton and for spending considerable
time in helpful conversations.

Finally, I would like to acknowledge the wonderful working environment at the IAS Princeton.  I have
benefitted from conversations here with many people.  I would particularly like to mention Abed
Bounemoura, Hakan Eliasson, \'Alvaro Pelayo, Kris Wysocki, and Edi Zehnder.

This work is based upon work supported by the National Science Foundation under agreement
No.\ DMS-0635607.  Any opinions, findings and conclusions or recommendations in this
material are those of the author and do not necessarily reflect the views of the National
Science Foundation.

\section{Preliminaries}\label{S:preliminaries}
We will study an irrational pseudo-rotation $\varphi\in\Diff^\infty(D,\omega_0)$ as generated by
a time-dependent Hamiltonian $H\in C^\infty(\R/\Z\times D,\R)$.  This means
the following.  For each $t\in\R/\Z$ abbreviate
\[
			      H^t:=H(t,\cdot)\in C^\infty(D,\R)
\]
and let $X_{H^t}$ be the unique $C^\infty$-smooth vector field on $D$ satisfying
\[
		      \omega_0(X_{H^t}(\xi),\cdot)=-dH^t(\xi)
\]
at $\xi\in D$.  If this vector field is tangent to the boundary of the disk then it generates a
$1$-parameter family of symplectic diffeomorphisms
\[
 \R\ni t\mapsto\phi_H^t\in\Diff^\infty(D,\omega_0)
\]
where for each $\xi\in D$, the curve $t\mapsto\phi_H^t(\xi)$ is the unique solution to
\[
 \left\{\begin{aligned}
	  \frac{d}{dt}\phi_H^t(\xi)&=X_{H^t}(\phi_H^t(\xi))\\
		      \phi_H^0(\xi)&=\xi.
        \end{aligned}\right.
\]
In particular $\phi_H^0=\id_D$.  We say that the Hamiltonian $H$ generates the disk map $\varphi$
if $\phi_H^1=\varphi$.  It is well known that any element of $\Diff^\infty(D,\omega_0)$ can be
generated in this manner by some suitable Hamiltonian $H$\!
\footnote{If $\psi\in\Diff^\infty(D,\omega_0)$ is the identity map
on the boundary up to first order then an isotopy from $\psi$ to $\id_{D}$ in $\Diff^\infty(D)$
can be converted to a symplectic isotopy using a Moser type argument.
For a general map $\psi\in\Diff^\infty(D,\omega_0)$ the first step therefore is to symplectically
``straighten up'' at the boundary.  That is, find a symplectic isotopy from $\psi$
to a map $\psihat\in\Diff^\infty(D,\omega_0)$  where $\psihat$ is the identity map
on the boundary up to first order.  That is, for all $\xi\in\partial D$
$\psihat(\xi)=\xi$ and $D\psihat(\xi)=\id_{\R^2}$.  This is easy to arrange.}.
In particular, any irrational pseudo-rotation.

So suppose that $H\in C^\infty(\R/\Z\times D,\R)$ generates a given irrational pseudo-rotation
$\varphi\in\Diff^\infty(D,\omega_0)$.  By assumption $\varphi$ has a unique fixed point $0\in D$,
and without loss of generality we may find $H$ so that $0\in D$ is a rest point for each vector
field $X_{H^t}$, $t\in\R/\Z$.  Thus for each $n\in\N$ the unique $n$-periodic solution
$\gamma:\R\rightarrow D$ to $\gammadot(t)=X_{H^t}(\gamma(t))$ is the constant trajectory
$\gamma(t)\equiv 0\in D$.

The map $\varphi:D\rightarrow D$ restricts to an orientation preserving diffeomorphism on the
boundary circle, $\varphi|_{\partial D}:\partial D\rightarrow\partial D$.
By assumption $\varphi|_{\partial D}$ has no periodic points and therefore its rotation number
$\Rot(\varphi)\in\R/\Z$ is irrational.  Since we have also fixed a Hamiltonian $H$ generating
$\varphi$, this allows us to associate a preferred real number, that we will denote by
\[
				  \Rot(\varphi;H)\in\R,
\]
with the property that the induced element on the circle $[\Rot(\varphi;H)]\in\R/\Z$ is equal
to $\Rot(\varphi)$.  See definition 5 in \cite{Bramham_approx}.

Let $\R^{+}=[0,\infty)$.  For each $n\in\N$, let
\[
 \M(H,n)
\]
denote the set of solutions
$z\in C^\infty(\R^{+}\times\R/n\Z,D)$ to the Floer equation
\begin{equation}\label{E:Floer_eqn}
      \partial_sz(s,t)+i\Big(\partial_tz(s,t)-X_{H^t}(z(s,t))\Big)=0
\end{equation}
for all $(s,t)\in\R^+\times\R/n\Z$, where $i$ is the standard complex structure on $D$
inherited from the complex plane,  which satisfy the following asymptotic and boundary conditions:
\begin{equation}\label{E:boundary_condns}
  \left\{\begin{aligned}
		  &\lim_{s\rightarrow\infty} z(s,t)=0& &\mbox{ for all }t\in\R/n\Z  \\
		  &z(0,t)\in\partial D& &\mbox{ for all }t\in\R/n\Z \\
		  &z(0,\cdot):\R/n\Z\rightarrow\partial D& &\mbox{ has degree }\lfloor n\alpha\rfloor.
	   \end{aligned}\right.
\end{equation}
Note that the asymptotic condition that the loops $z(s,\cdot)\rightarrow 0$ in $C^0(\R/n\Z,D)$ as
$s\rightarrow\infty$ is equivalent to finiteness of the Floer energy of $z$:
\begin{equation}\label{E:Floer_energy}
	E_{\Floer}(z):=\int_{s=0}^{\infty}\int_{t=0}^n\big|z_s(s,t)\big|^2+\big|z_t(s,t)-X_{H^t}(s,t)\big|^2ds dt,
\end{equation}
because there is only one $n$-periodic orbit of the Hamiltonian vector field.
The norm in the integrand is from the standard Euclidean
metric.  Similarly, in the rest of the paper, all Sobolev and $C^r$ function spaces for maps into
$D\subset\R^2$ are with respect to the standard Euclidean norm on $\R^2$.

The existence result we use is the following.

\begin{theorem}\label{T:M(n,k)}
For each $n\in\N$ the space of solutions $\M(H,n)$ is
non-empty, and the following holds:
\begin{itemize}
 \item\textbf{Filling property:} For all $p\in D\backslash\{0\}$ there exists a solution $z\in\M(H,n)$
such that $p$ lies in the image of $z$.
 \item$\mathbf{L^2}$\textbf{-estimates:}  For all $z\in\M(H,n)$,
 \begin{equation}\label{E:L2_estimates}
	    \|\partial_sz\|^2_{L^2([0,\infty)\times\R/n\Z)}=\{n\alpha\}\pi
\end{equation}
where for $x\in\R$, $\{x\}\in[0,1)$ denotes its fractional part.
 \item$\mathbf{C^\infty}$\textbf{-bounds:} Let $n_{j}$ be a subsequence for which
 $\{n_{j}\alpha\}\rightarrow0$ as $j\rightarrow\infty$.  For all $r\in\N$, there exists
 $b_r\in(0,\infty)$ such that
 \begin{equation}\label{E:Cr_estimates}
		      \|\nabla z\|_{C^r([0,\infty)\times\R/n_{j}\Z)}\leq b_r
\end{equation}
for all $z\in\M(H,n_{j})$, uniformly in $j$.
\end{itemize}
\end{theorem}

These solutions arise from a foliation of the $4$-manifold $\R\times \R/n\Z\times D$ by pseudoholomorphic
curves; the disk component of each solution to the Cauchy-Riemann equation satisfies this Floer equation
if the former has finite energy (Gromov's trick in reverse).  We explain this in section
\ref{S:existence_of_solutions}, and thus how theorem \ref{T:M(n,k)} follows from a
construction in \cite{Bramham_approx}.

\begin{remark}
The number $\alpha$ in theorem \ref{T:M(n,k)} can be any irrational; no Liouville condition is
assumed at this point.
\end{remark}

To use this to prove our main result we will need two lemmas.  The
first is the following Sobolev type inequality.
\begin{lemma}\label{L:sobolev_estimate}
Fix $d\in\N$.  There exists $c>0$, so that for all $n\in\N=\{1,2,\ldots\}$,
\[
   \|f\|^2_{L^\infty}\leq c\|f\|_{L^2}\|f\|_{W^{1,\infty}}
\]
for all $f\in C_c^\infty(\R^+\times\R/n\Z,\R^d)$.  We emphasize that $c$ is independent of $n$.
\end{lemma}
This is proven in appendix \ref{S:sobolev_estimate}.  The second lemma we require is:
\begin{lemma}\label{L:Gronwall}
Suppose that $x:[0,T]\rightarrow[0,\infty)$ is a continuous function, some $T\geq 0$, for which
there exist constants $a,b\geq0$ such that for all $t\in[0,T]$
\begin{equation}\label{E:Gronwall_assumption_1}
		    x(t)\leq a + b\int_0^tx(s)ds.
\end{equation}
Then
\begin{equation}\label{E:Gronwall_conclusion_1}
			x(t)\leq ae^{bt}
\end{equation}
for all $t\in[0,T]$.
\end{lemma}
This is a version of the familiar Gronwall inequality.  For a proof see for example lemma 6.1 in \cite{Amann}.

\section{Proof of $C^0$-rigidity}\label{S:rigidity}

In this section we prove theorem \ref{T:main_thm}.  Recall that this said the following:
\begin{theorem}\label{T:main_thm_2}
If $\varphi:D\rightarrow D$ is a pseudo-rotation with rotation number in $\L_*$ then
$\varphi^{n_j}\rightarrow\id_D$ in the $C^{0}$-topology, for some sequence of integers
$n_j\rightarrow+\infty$.
\end{theorem}

\begin{remark}
Obviously a necessary condition for $\varphi$ in theorem \ref{T:main_thm_2} to be rigid in this sense
is that its
restriction to the boundary $\varphi|_{\partial D}:\partial D\rightarrow\partial D$ be rigid.
Indeed any sufficiently smooth (e.g.\ $C^2$) orientation preserving
diffeomorphism $f:\partial D\rightarrow\partial D$ with irrational rotation
number is topologically conjugate to a rigid rotation and therefore $C^0$-rigid.
Clearly this line of argument will not apply for the general pseudo-rotations as any ergodic example
cannot be conjugated to a rigid rotation.
\end{remark}

Our main tool in the proof is theorem \ref{T:M(n,k)} from the previous section.  To use this
we fix a Hamiltonian $H\in C^\infty(\R/\Z\times D,\R)$ whose time-one map is the given
pseudo-rotation $\varphi$.  By assumption the rotation number of the circle map $\varphi|_{\partial D}$
is an element of $\L_*/\Z\subset\R/\Z$.  With respect to $H$ we have a canonical lift
\[
			    \alpha:=\Rot(\varphi;H)\in\L_*\subset\R.
\]
As $\alpha\in\R$ is irrational every point in the interval $[0,1]$ is an
accumulation point of its sequence of fractional parts $\{\{n\alpha\}\}_{n\in\N}$.  In particular
zero is and so there exists a subsequence $n_j\rightarrow\infty$ such that $\{n_j\alpha\}\rightarrow0$
as $j\rightarrow\infty$.

Since moreover $\alpha$ belongs to the subset of Liouville numbers $\L_*$ we know that for every $j\in\N$
there exists $(p_j,n_j)\in\Z\times\N$ such that
\begin{equation*}
		0<\left|\alpha - \frac{p_j}{n_j}\right|<\frac{1}{e^{jn_j}}.
\end{equation*}
Taking a further subsequence we may assume that
\begin{equation}\label{E:approach_1}
  \{n_j\alpha\}\leq\frac{n_j}{e^{jn_j}}
\end{equation}
for all $j\in\N$.

\begin{remark}
It is apriori possible that we can only conclude that $1-\{n_j\alpha\}\leq\frac{n_j}{e^{jn_j}}$
for all $j\in\N$.  In this case there are two ways to proceed.  Simplest is to replace
$\varphi$ by its inverse $\varphi^{-1}$.  Then (\ref{E:approach_1}) will hold where $\alpha$ is the
rotation number of $\varphi^{-1}$, and so the remainder of
our arguments will show that $\varphi^{-1}$ is $C^0$-rigid.  Which is equivalent to $\varphi$ being
$C^0$-rigid.  A more natural way to handle this possibility is to work with a version
of theorem \ref{T:M(n,k)} for solutions $z$ of the Floer equation for which the
degree of the map $z(0,\cdot):\R/n\Z\rightarrow\partial D$ is $\lceil n\alpha\rceil$, compared
with (\ref{E:boundary_condns}).  Then the remaining arguments would apply to $\varphi$ directly.
\end{remark}

Recall that, given the Hamiltonian $H$, we defined for each $n\in\N$
the moduli space $\M(H,n)$, see (\ref{E:Floer_eqn}) and (\ref{E:boundary_condns}).
In particular, if $z_j\in C^\infty(\R^+\times\R/n_j\Z,D)$ belongs to $\M(H,n_j)$ then
\begin{align}
         &\partial_sz_j+i\big(\partial_tz_j-X_{H^t}(z_j)\big)\equiv 0\label{E:pde}\\
         &\|\partial_sz_j\|^2_{L^2}=\{n_j\alpha\}\pi\label{E:L2_bds}\\
         &\|\partial_sz_j\|_{C^1}\leq b\label{E:C1_bds}
\end{align}
for some constant $b\in(0,\infty)$ depending only on the Hamiltonian $H$.

\begin{proof}[Proof of theorem \ref{T:main_thm_2}]
We break this down into three steps.  All functional norms are on the entire domain,
e.g.\ $\|\partial_sz_j\|_{L^2}=\|\partial_sz_j\|_{L^2(\R^+\times\R/n_j\Z)}$.

\textbf{Step 1:}
By the interpolation inequality lemma \ref{L:sobolev_estimate} there exists
$c\in(0,\infty)$ independent of everything, so that
\begin{align*}
	   \|\partial_sz_j\|^2_{L^\infty}&\leq c\|\partial_sz_j\|_{L^2}\|\partial_sz_j\|_{W^{1,\infty}}
\intertext{for all $z_j\in\M(H,n_j)$.  So by (\ref{E:L2_bds}) and (\ref{E:C1_bds}),}
	    \|\partial_sz_j\|^2_{L^\infty}&\leq c\{n_j\alpha\}^{1/2}\pi^{1/2}b.
\end{align*}
Thus,
\begin{equation}
	\|\partial_sz_j\|_{L^\infty}\leq M\{n_j\alpha\}^{1/4}
\end{equation}
where $M=(cb)^{1/2}\pi^{1/4}$ is a constant depending only on the loop of Hamiltonians $H^t$.

\textbf{Step 2:}
Let $p\in D\backslash\{0\}$.  For each $j$ there exists a solution $z_j\in\M(H,n_j)$
whose image contains $p$.  This was from theorem \ref{T:M(n,k)}.  Thus, after reparameterizing
if necessary, $z_j(s,0)=p$ for some $s=s_{j}\in\R^+$.

Let $t\mapsto\phi_{H}^t(p)$ be the $1$-parameter family of
diffeomorphisms on the disk generated by the loop of Hamiltonians $H^t$.  For all
$t\in\R$,
\[
	z_j(s,t)-\phi_{H}^t(p)=\int_0^t\partial_{\tau}z_j(s,\tau)-X_{H^{\tau}}(\phi_{H}^{\tau}(p)) d\tau.
\]
Thus,
\begin{align*}
  |z_j(s,t)-\phi_{H}^t(p)|\leq&\int_0^t|\partial_{\tau}z_j(s,\tau)-X_{H^{\tau}}(\phi_{H}^{\tau}(p))| d\tau\\
			\leq&\int_0^t|\partial_{\tau}z_j(s,\tau)-X_{H^{\tau}}(z_j(s,\tau))| d\tau+{}\\
			& {}+\int_0^t|X_{H^{\tau}}(z_j(s,\tau))-X_{H^{\tau}}(\phi_{H}^{\tau}(p))| d\tau\\
			\leq&\int_0^t|\partial_sz_j(s,\tau)| d\tau+
				\int_0^t\|DX_{H^{\tau}}\||z_j(s,\tau))-\phi_{H}^{\tau}(p)| d\tau\\
			\leq&A_jt + B\int_0^t|z_j(s,\tau))-\phi_{H}^{\tau}(p)| d\tau
\end{align*}
where $B=\max_{\tau\in\R/\Z}\max_{\xi\in D}\|\Hess(H^{\tau})(\xi)\|$ and
\[
 A_j=\|\partial_sz_j\|_{L^\infty}.
\]
Applying the Gronwall inequality lemma \ref{L:Gronwall},
\begin{equation*}
	      |z_j(s,t)-\phi_{H}^t(p)|\leq A_jte^{Bt}
\end{equation*}
for all $t\geq0$.  In particular, as $z_j$ is $n_j$-periodic in the $t$-variable,
$|p-\varphi^{n_j}(p)|=|z_j(s,0)-\phi_{H}^{n_j}(p)|=|z_j(s,n_j)-\phi_{H}^{n_j}(p)|\leq A_jn_je^{Bn_j}$.
By step 1 we have an estimate on $A_j$, from which we obtain
\begin{equation*}
	      |p-\varphi^{n_j}(p)|\leq M\{n_j\alpha\}^{1/4}n_je^{Bn_j}.
\end{equation*}
The right hand side is independent of $p\in D\backslash\{0\}$, and so
\begin{equation}\label{E:distance_to_id}
	      d_{C^0}(\varphi^{n_j},\id_{D})\leq M\{n_j\alpha\}^{1/4}n_je^{Bn_j}
\end{equation}
for all $j\in\N$.

\textbf{Step 3:} It remains to use the Liouville condition on $\alpha$.  Substituting
(\ref{E:approach_1}) into (\ref{E:distance_to_id}) we obtain
\begin{equation*}\label{E:distance_to_id_2}
	      d_{C^0}(\varphi^{n_j},\id_{D})\leq M \frac{n_j^{1/4}}{e^{jn_j/4}} n_je^{Bn_j}.
\end{equation*}
The right hand side decays to zero because $B$ is finite.
This completes the proof of theorem \ref{T:main_thm_2}.
\end{proof}

\section{Proof of theorem \ref{T:M(n,k)}}\label{S:existence_of_solutions}
In this section we explain how theorem \ref{T:M(n,k)} follows from the existence of
certain finite energy foliations constructed in \cite{Bramham_approx}.

\subsection{From the disk to mapping tori}
On the symplectic manifold $(D,\omega_0=dx\wedge dy)$ we have a time-dependent Hamiltonian
$H\in C^\infty(\R/\Z\times D,\R)$ generating a given irrational pseudo-rotation
$\varphi:D\rightarrow D$.  That is, $\varphi$ is the time-one map of the $1$-parameter
family of symplectic diffeomorphisms generated by a path of Hamiltonian vector
fields $X_{H^t}$ on the disk, as described in section \ref{S:preliminaries}.

Consider the autonomous vector field on the infinite tube $Z_\infty:=\R\times D$;
\[
		      R(\tau,z):=\partial_\tau + X_{H^\tau}(z)
\]
at $(\tau,z)\in\R\times D$.  The transformation $\T:Z_\infty\rightarrow Z_\infty$,
$(\tau,z)\mapsto(\tau-1,z)$ leaves $R$ invariant due to the periodicity of $H^t=H(t,\cdot)$ in
the $t$-variable.  Therefore, for each $n\in\N$, $R$ descends
to a vector field $R_n$ on the quotient space
\[
		    Z_n:=\R/n\Z\times D,
\]
and by a slight abuse of notation we will write $(\tau,z)$ for coordinates on $Z_n$.
The disk slice $\{0\}\times D\subset Z_n$ is a global Poincar\'e section to the flow
of $R_n$ and the first return map is the $n$-th iterate of
the pseudo-rotation $\varphi^n:D\rightarrow D$.

\subsection{To almost complex manifolds}
To each mapping torus $(Z_n,R_n)$, $n\in\N$, we associate an almost complex $4$-manifold
$(W_n,J_n)$ as follows.  Set
\[
			  W_n:=\R\times Z_n=\R\times\R/n\Z\times D
\]
equipped with coordinates $(a,\tau,z)$.  Let $J_n$ be the unique almost complex structure on $W_n$
characterized by the conditions
\begin{equation*}\label{E:almost_complex_structure}
  \left\{\begin{aligned}
          &J_n(a,\tau,z)\partial_\R=R_n(\tau,z)\\
          &J_n(a,\tau,z)|_{TD}=i
         \end{aligned}\right.
\end{equation*}
for all $(a,\tau,z)\in W_n$, where $\partial_\R$ is the vector field dual to the $\R$-coordinate on $W_n$.
Recall that $i$ is the standard complex structure on $D$ as a subspace of the complex plane $\C$.
Consider the $2$-tori
\[
		L_c:=\{c\}\times\partial Z_n=\{c\}\times\R/n\Z\times\partial D
\]
over $c\in\R$, which foliated the boundary of $W_n$.  Each $L_c$ is totally real with respect to $J_n$.
That is, $T_{p}L_c\oplus J_n(p)(T_{p}L_c)=T_pW_n$ for all $p\in L_c$.

\subsection{The pseudoholomorphic curves}
For each $n\in\N$ let
\[
 \M(J_n)
\]
denote the set of solutions $\u=(a,\tau,z)\in C^\infty(\R^+\times\R/n\Z,W_n)$ to the Cauchy-Riemann equation
\begin{equation}\label{E:CR_eqn_2}
		  \partial_s\u(s,t)+J_n(\u(s,t))\partial_t\u(s,t)=0
\end{equation}
for all $(s,t)\in\R^+\times\R/n\Z$, with the following boundary conditions: there exists $c\in\R$
such that
\begin{equation*}\label{E:boundary_condns_2}
  \left\{\begin{aligned}
		  &\u(0,t)\in L_c& &\mbox{ for all }t\in\R/n\Z \\
		  &z(0,\cdot):\R/n\Z\rightarrow\partial D& &\mbox{ has degree }\lfloor n\alpha\rfloor \\
		  &\tau(0,\cdot):\R/n\Z\rightarrow\partial\R/n\Z& &\mbox{ has degree }+1,
	   \end{aligned}\right.
\end{equation*}
and which additionally satisfy the following finite energy conditions
\[
		    E_\lambda(\u)<\infty\qquad\mbox{and}\qquad E_\omega(\u)<\infty
\]
which we explain now.
If $\u$ is a solution to (\ref{E:CR_eqn_2}) then the following integrals have well defined values in
$\R^+\cup\{+\infty\}$.
\begin{equation*}\label{E:lambda_energy}
	E_{\lambda}(\u):=\sup_{\psi}\int_{\R^+\times\R/n\Z}\u^*\Big(\psi(a)da\wedge d\tau\Big)
\end{equation*}
where the supremum is taken over all $\psi\in C^\infty(\R,\R^+)$ for which
$\int_{-\infty}^{\infty}\psi(s)ds=1$.
\begin{equation*}\label{E:omega_energy}
	E_{\omega}(\u):=\int_{\R^+\times\R/n\Z}\u^*\omega_n
\end{equation*}
where $\omega_n$ is the differential $2$-form $\omega_n:=dx\wedge dy + d\tau\wedge dH$
on $Z_n$.  Indeed, the pull-back $2$-forms $\u^*\Big(\psi(a)da\wedge d\tau\Big)$ and
$\u^*\omega_n$ are pointwise non-negative multiples of $ds\wedge dt$ on the domain $\R^+\times\R/n\Z$.

\begin{remark}
These two energies $E_{\lambda}(\u)$ and $E_{\omega}(\u)$ come from the compactness theory in
\cite{BEHWZ_SFT_compactness} developed for symplectic field theory where they are defined for
pseudoholomorphic maps in much more general settings.
In particular if $M$ is an oriented compact $3$-manifold equipped with a pair of differential forms
$(\omega,\lambda)$ with the following properties: (1) $\omega$ is a closed $2$-form, $\lambda$ is a $1$-form,
(2) $\lambda\wedge\omega>0$, (3) that $\ker(\omega)\subset\ker(d\lambda)$.  Such a pair $(\omega,\lambda)$
is called a \emph{stable Hamiltonian structure} on $M$, see \cite{Cieliebak_Volkov} for examples and properties.
Then there is a suitable class of so called cylindrical, symmetric almost complex structures on $\R\times M$
which satisfy a compatibility condition with
$\omega$ and $\lambda$.  With respect to such almost complex structures there are two notions
of energy for a pseudoholomorphic curve $\v$, commonly written $E_\omega(\v)$ and $E_{\lambda}(\v)$.
The spaces of curves with a uniform bound on both of these energies enjoys a nice compactness theory.
In this paper $M$ is $Z_n=\R/n\Z\times D$ for any $n\in\N$, and the stable Hamiltonian structure
on $Z_n$ is $(\omega_n,\lambda_n)$ where $\omega_n$ is the $2$-form above and $\lambda_n=d\tau$.
\end{remark}

\subsection{To Floer trajectories}
Recall the following relation between pseudoholomorphic curves with finite
$\lambda$-energy and solutions to Floer's equation.  For a proof see lemma 6 in \cite{Bramham_approx}.

\begin{lemma}\label{L:CR_Floer_relation}
Let $n\in\N$ and suppose $\u\in C^\infty(\R^+\times\R/n\Z,W_n)$.  Then
$\u\in\M(J_n)$ if and only if
\[
		      \u(s+s_0,t+t_0)=(s,t,z(s,t))
\]
for some $z\in\M(H,n)$ and $(s_0,t_0)\in\R\times\R/n\Z$.  The energies are then related by
\begin{equation}\label{E:energies}
  \left\{\begin{aligned}
		  &E_\lambda(\u)=n\\
		  &E_\omega(\u)=E_{\Floer}(z),
	   \end{aligned}\right.
\end{equation}
where $E_{\Floer}(z)$ is the energy from Floer theory defined by (\ref{E:Floer_energy}).
\end{lemma}

Because of this lemma, theorem \ref{T:M(n,k)} which is a statement about the spaces
$\M(H,n)$, is an immediate consequence of the following
statement about the spaces $\M(J_n)$.

\begin{theorem}
Let $\varphi$ be an irrational pseudo-rotation, and $H$ a generating Hamiltonian.
Let $\alpha:=\Rot(\varphi;H)$ be the real valued rotation number on the boundary.
So $\alpha\in\R$ is irrational.  Then for each $n\in\N$ there exists a foliation
$\F_n$ of $W_n=\R\times Z_n$ by surfaces which can be described as follows:
\begin{itemize}
 \item The cylinder
 \[
		C_n:=\big\{(a,\tau,0)\in\R\times Z_n\,|\, a\in\R,\ \tau\in\R/n\Z  \big\}
 \]
 is a leaf in $\F_n$.
 \item Each leaf in $\F_n$ besides $C_n$ is the image of a solution $\u\in\M(J_n)$
 satisfying
 \begin{equation}\label{E:omega_energy_for_u}
			E_\omega(\u)=\{n\alpha\}\pi.
 \end{equation}
 \item Let $n_{j}$ be a subsequence for which $\{n_{j}\alpha\}\rightarrow0$ as $j\rightarrow\infty$.
 For all $r\in\N$ there exists $B_r\in(0,\infty)$ such that for all $\u\in\M(J_{n_{j}})$,
 \begin{equation}\label{E:grad_bds_for_u}
		  \|\nabla\u\|_{C^r(\R^+\times\R/n_{j}\Z)}\leq B_r
 \end{equation}
 uniformly in $j$.
\end{itemize}
\end{theorem}
\begin{proof}
The existence of each $\F_n$ as a foliation whose leaves consist of the cylinder $C_n$ and
half cylinders parameterized by elements of $\M(J_n)$ is immediate from
theorem 8 in \cite{Bramham_approx}.  That each $\u\in\M(J_n)$ satisfies
$E_\omega(\u)=\{n\alpha\}\pi$ is lemma 9 in \cite{Bramham_approx}.  Finally the uniform $C^r$-bounds
(\ref{E:grad_bds_for_u}) follow from corollary 24 and proposition 22 in \cite{Bramham_approx}.
\end{proof}

\appendix
\section{}

\subsection{A Sobolev inequality}\label{S:sobolev_estimate}
In this section we prove lemma \ref{L:sobolev_estimate}.
We first consider maps with domain the half plane $\R^{+}\times\R$, and then modify for cylinders
$\R^{+}\times\R/n\Z$.  It is important for us that the Sobolev constant be independent of the
period $n$ of the domain.

All function spaces will implicitely mean real valued functions, e.g.\
$L^\infty(\R^2)=L^\infty(\R^2,\R)$.  Then lemma \ref{L:sobolev_estimate} for maps into $\R^{d}$
will follow by applying the conclusions to each component.  As usual, $C_c^\infty(\Omega)$ means
elements of $C^{\infty}(\Omega)$ having compact support.

\begin{lemma}\label{L:sobolev_estimate_1}
There exists $C>0$ so that
\[
	\|f\|^2_{L^\infty(\R^2)}\leq C \|f\|_{L^2(\R^2)}\|Df\|_{L^\infty(\R^2)}
\]
for all $f\in C_c^\infty(\R^2)$.  The same statement holds if we replace
$\R^2$ by $\R^{+}\times\R$ with no boundary conditions required.
\end{lemma}
\begin{proof}
Let $f\in C_c^\infty(\R^2)$, $p\geq0$ a real number, and $(x,y)\in\R^2$.  Then
\begin{align}
		|f(x,y)|^{p+1}&\leq \int_{-\infty}^x|\partial_1(f(s,y))^{(p+1)}|ds\nonumber \\
		&\leq (p+1)\|\partial_1f\|_{L^\infty(\R^2)}\int_{\R}|f(s,y)|^p ds.\label{E:222}
\end{align}
Applying this with $p=3$,
\begin{equation}\label{E:223}
		|f(x,y)|^{4}\leq 4\|Df\|_{L^\infty(\R^2)}\int_\R|f(s,y)|^3 ds.
\end{equation}
While applying (\ref{E:222}) with $p=2$,
\begin{equation}\label{E:224}
		|f(s,y)|^{3}\leq 3\|Df\|_{L^\infty(\R^2)}\int_\R|f(s,t)|^2 dt
\end{equation}
for each $s\in\R$.  Substituting (\ref{E:224}) into (\ref{E:223}),
\begin{align*}
  |f(x,y)|^{4}&\leq 12\|Df\|^2_{L^\infty(\R^2)}\int_\R\left(\int_\R|f(s,t)|^2 dt\right)ds \\
	      &=12\|Df\|^2_{L^\infty(\R^2)}\|f\|^2_{L^2(\R^2)}.
\end{align*}
Square rooting both sides we are done.  When
the domain is $\R^{+}\times\R$ the same argument goes through almost word
for word (this time integrating from $+\infty$ to $x$ to obtain equation (\ref{E:222})).
\end{proof}

\begin{lemma}\label{L:sobolev_estimate_2}
There exists $c>0$ so that for all $n\geq 1$,
\[
   \|f\|^2_{L^\infty(\R\times\R/n\Z)}\leq c \|f\|_{L^2(\R\times\R/n\Z)}\|f\|_{W^{1,\infty}(\R\times\R/n\Z)}
\]
for all $f\in C_c^\infty(\R\times\R/n\Z)$.  The same statement holds if we replace $\R\times\R/n\Z$
by $\R^{+}\times\R/n\Z$ with no boundary conditions required.
\end{lemma}
\begin{proof}
Let $f\in C_c^\infty(\R\times\R/n\Z)$.  Let $\fbar\in C^\infty(\R^2)$ be a
lift of $f$.  Then
\[
		    \fbar(x,y+n)=\fbar(x,y)
\]
for all $(x,y)\in\R^2$, and $\lim_{x\rightarrow\pm\infty}\fbar(x,y)=0$
for all $y\in\R$.  Pick a smooth ``cut-off function'' $\chi:\R\rightarrow[0,1]$
having support in $(-1,n+1)$, and identically equal to $1$ on $[0,n]$,
and such that $|\chi'(t)|\leq 2$ for all $t\in\R$.   Define $g\in C_c^\infty(\R^2)$ by
\[
		  g(x,y):=\chi(y)\fbar(x,y).
\]
Then $g|_{\R\times[0,n]}\equiv\fbar|_{\R\times[0,n]}$.  Let $C>0$ be
the embedding constant from lemma \ref{L:sobolev_estimate_1}.  Then,
\[
	\|g\|^2_{L^\infty(\R^2)}\leq C \|g\|_{L^2(\R^2)}\|Dg\|_{L^\infty(\R^2)}.
\]
Therefore,
\begin{align*}
	\|\fbar\|^2_{L^\infty(\R\times[0,n])}&=\|g\|^2_{L^\infty(\R\times[0,n])} \nonumber \\
			    &\leq C \|g\|_{L^2(\R^2)}\|Dg\|_{L^\infty(\R^2)} \nonumber \\
	 &\leq C\|\fbar\|_{L^2(\R\times[-n,2n])}
	    \left(\|\chi'\fbar\|_{L^\infty(\R^2)}+ \|\chi (D\fbar)\|_{L^\infty(\R^2)} \right)
		  \label{E:231}
\intertext{using that $n\geq1$,}
	 &\leq C 3\|\fbar\|_{L^2(\R\times[0,n])}
	    \left(2\|\fbar\|_{L^\infty(\R^2)}+ \|D\fbar\|_{L^\infty(\R^2)} \right)\\
	 &\leq 6C \|f\|_{L^2(\R\times\R/n\Z)}
	    \left(\|f\|_{L^\infty(\R\times\R/n\Z)}+ \|Df\|_{L^\infty(\R\times\R/n\Z)} \right). \nonumber
\end{align*}
In other words,
\[
    \|f\|^2_{L^\infty(\R\times\R/n\Z)}
	    \leq 6C \|f\|_{L^2(\R\times\R/n\Z)}\|f\|_{W^{1,\infty}(\R\times\R/n\Z)}
\]
as required.  The same argument applies when the domain is $\R^{+}\times\R/n\Z$.
\end{proof}

\subsection{$\L_*$ is dense in $\R$}\label{S:density_L_0}
We refer to (\ref{E:exponential_Liouville_condition}) in the introduction for the definition of $\L_{*}$.
The following is a well known argument for spaces like $\L_{*}$.

\begin{lemma}
$\L_*$ is a dense subset of $\R$.
\end{lemma}
\begin{proof}
For each $(p,q)\in\Z\times\N$ relatively prime, and $k\in\N$, define the following punctured open neighborhood
of $p/q\in\R$:
\[
	\O_k(p,q):=\left\{x\in\R\,\Big|\,0<\left|x-\frac{p}{q}\right|<\frac{1}{e^{k q}}\right\}.
\]
For each $k\in\N$ the following countable union is therefore also an open subset of $\R$,
\[
	      \U_k:=\bigcup_{(p,q)\in\Z\times\N,\ (p,q)=1}\O_k(p,q).
\]
Fix $k\in\N$.  Then each rational number $\omega\in\Q$ lies in the closure of $\U_k$ since if $\omega=p/q$ in
lowest form then $\omega$ lies in the closure of $\O_k(p,q)$.  Thus, as the rationals are dense in $\R$,
$\U_k$ must be dense in $\R$.  This applies to all $k\in\N$, so
\[
		  \L_*=\bigcap_{k\in\N}\U_k
\]
is a countable intersection of open dense subsets of the complete metric space $\R$.  By the
Baire category theorem $\L_*$ is therefore dense in $\R$.
\end{proof}

\end{document}